\documentclass[12pt,reqno]{amsart}

\usepackage{amsmath,amssymb,amscd,graphicx, amsthm, 
  enumerate, hyperref, color}
 \usepackage{amsfonts, amsthm, amssymb, color, extarrows, enumitem, nicefrac, blindtext, mathtools}
\calclayout

\usepackage{mathrsfs}
\usepackage{mathtools}
\usepackage{amsmath}
\usepackage{color,xcolor}
\definecolor{darkred}{rgb}{1,0,0}
\definecolor{darkgreen}{rgb}{0,0.8,0}
\definecolor{darkblue}{rgb}{0,0,1}
\hypersetup{colorlinks,
linkcolor=darkblue,
filecolor=darkgreen,
urlcolor=darkred,
citecolor=darkgreen}
\numberwithin{equation}{section}
\theoremstyle{plain}
\newtheorem{theorem}{Theorem}
\numberwithin{theorem}{section}
\newtheorem{proposition}[theorem]{Proposition}
\newtheorem{lemma}[theorem]{Lemma}

\newtheorem{corollary}[theorem]{Corollary}
\theoremstyle{definition}
\newtheorem{remark}[theorem]{Remark}
\newtheorem{definition}[theorem]{Definition}

\newtheorem{example}[theorem]{Example}

  \newcommand{\re}{\mathbb{R}}

\newcommand{\boxx}{\rule{2.12mm}{3.43mm}}

\title[Unique continuation for a gradient inequality]{Unique continuation for a gradient inequality with $L^n$ potential}

\author{Adam Coffman, \ \ Yifei Pan,\ \  and \ \  Yuan Zhang}
\address{Department of Mathematical Sciences, Purdue University Fort Wayne, 2101 East Coliseum Boulevard, Fort Wayne, IN 46805, USA}
\email{coffmana@pfw.edu, pan@pfw.edu, zhangyu@pfw.edu}

\begin{document}

 \begin{abstract}
We establish a unique continuation property for solutions of the differential inequality $|\nabla u|\leq V|u|$, where $V$ is locally $L^n$ integrable on a domain in $\mathbb R^n$. A stronger uniqueness result is obtained if in addition the solutions are locally Lipschitz. One application is a finite order vanishing property in the $L^2$ sense for the exponential of $W^{1,n}$ functions. We further discuss related results for the Cauchy-Riemann operator $\bar\partial$ and characterize the vanishing order for smooth extension of holomorphic functions across the boundary.  
\end{abstract}

\maketitle

\section{Introduction and Results}

Let $\Omega$ be a connected open subset of $\mathbb R^n$.  We investigate   solutions to the following differential inequality concerning the gradient operator $\nabla$:
\begin{equation}\label{eq}
     |\nabla u| \leq V |u | \ \ \text{on}\ \ \Omega,
\end{equation}
with the potential $V\in L^n_{loc}(\Omega)$. 

A    function $u\in L^2_{loc}(\Omega)$ is said to vanish to infinite order (or to be {\underline{flat}})  at a point $x_0\in \Omega$  (in the $L^2$ sense) means that for all $m\ge 0$,
\begin{equation*}\label{flat}
    \lim_{r\rightarrow 0} r^{-m}\int_{|x-x_0|<r}|u(x)|^2 dv  =0,
\end{equation*}
where $dv$  is the Lebesgue  measure element in $\mathbb R^n$. Otherwise,  $u$ vanishes to finite order at $x_0$ in the $L^2$ sense.  We say a differential (in)equality  satisfies the (strong) unique continuation property to mean that every $H^1_{loc}(\Omega)\left(= W^{1,2}_{loc}(\Omega)\right)$ solution that vanishes to infinite order at a point in the $L^2$ sense must vanish identically. Here for  $p\ge 1$, $W^{1,p}_{loc}(\Omega)$ is the standard Sobolev space of  $L^p_{loc}(\Omega)$ functions whose first order weak derivatives are represented by functions in $L_{loc}^p(\Omega)$.  While studying  the unique continuation property of the Cauchy-Riemann operator $\bar\partial$ in several complex variables: \begin{equation}\label{eqc}
    |\bar\partial u| \le V|u|\ \ \ \text{on}\ \ \Omega\subseteq\mathbb C^n,   
\end{equation} we observe that  \eqref{eqc} is  reduced to \eqref{eq}   when the solutions are real-valued. This motivates us to study the following unique continuation property of $H^1_{loc}(\Omega) $ solutions to \eqref{eq}.

\begin{theorem}\label{main}
Let $\Omega$ be a  domain in $\mathbb R^n, n\ge 2$ and $V\in L_{loc}^n(\Omega)$. Suppose $u=(u_1, \ldots, u_M): \Omega\rightarrow \mathbb R^M$ with $u\in H^1_{loc}(\Omega)$  and satisfies $ |\nabla u|\le V|u|$ a.e.\ on $\Omega$.  If $u$ vanishes to infinite order at some $x_0\in \Omega$, then $u\equiv 0$.
\end{theorem}

 The $n=2$ case in Theorem \ref{main} is due to a unique continuation property result in \cite{PZ} concerning the $\bar\partial$ operator. For higher dimensions, the proof makes use of a Hardy-type inequality, along with the Gagliardo-Nirenberg-Sobolev inequality. When the potential is no  longer in $L^n$, one can still get the unique continuation property for some special types of potentials, see Theorem \ref{main2}. However, as shown in Example \ref{ex}, the property fails in general for $V\notin L^n$.   On the other hand, Theorem \ref{wucp} states that the weak unique continuation property always holds for \eqref{eq} as long as $V\in L^2$.

As a consequence of   Theorem \ref{main},  we  obtain the following property of vanishing to finite order for the exponential of $ W^{1,n}$ functions.  Note  that the  $W^{1,n}$ space is the critical Sobolev space where the Sobolev embedding theorem fails, and instead is substituted by the Moser-Trudinger inequality.

 \begin{theorem}\label{nf}
 Let $\Omega$ be an open set in $\mathbb{R}^n, n\ge 2$. Suppose   $\phi:\Omega\to\re$  with $\phi\in W_{loc}^{1,n}(\Omega)$. Then   the exponential $e^\phi$ of $\phi$  vanishes to  finite order in the $L^2$ sense at each point in $\Omega$.  
\end{theorem}

In the second part of the paper, we focus on locally Lipschitz solutions to \eqref{eq}. 
Under the context of  this  more restricted function space, we are able to prove a uniqueness result below by just assuming the vanishing of the first jets. A similar uniqueness result was discussed in \cite{PW} for higher order differential operators on smooth functions of one variable ($n=1$). It is worth pointing out that the Lipschitz assumption on the solutions cannot  be dropped here  when $n\ge 2$, see Remark \ref{rem}. 

\begin{theorem}\label{thm1.1}
  Let $\Omega$ be a domain in $\mathbb R^n, n\ge 1$ and $V \in L^n_{loc}(\Omega)$. Suppose  $u=(u_1, \ldots, u_M):\Omega\to\re^M$    is a locally Lipschitz function on $\Omega$  satisfying
   $|\nabla u|\leq V |u|$  a.e.\ on $\Omega$.
  If $u(x_0)=0$ at some $x_0\in \Omega$, then $u\equiv 0$.
\end{theorem}

 Theorem \ref{thm1.1} can be readily applied to study the uniqueness of some types of nonlinear differential systems, as indicated in Corollary \ref{coo}. In Section \ref{sec5}, we discuss further applications  under the Lipschitz setting.  In particular,  Theorem \ref{t1} shows that the logarithm of a positive Lipschitz function cannot fall in $ W^{1,n}$ near every zero point of the function.  On the other hand, we prove   that   if  in addition $e^\phi$    in Theorem \ref{nf} is Lipschitz, then $e^\phi$ must be nowhere zero, see  Corollary \ref{nf1}.

In the last  section, we  discuss related results for  the $\bar\partial$ operator on domains in $\mathbb C^n$. To start with, we construct Example \ref{ex5} to show that  the gradient operator $\nabla$ in Theorem \ref{thm1.1} cannot be replaced by the $\bar\partial$ operator even for real analytic functions. On the other hand, we  give finer characterizations  in terms of an $L^2$ divergence for holomorphic functions that are extended  smoothly across the boundary.



\section{Unique continuation for $H^1$ solutions}
Let $\Omega$ be a domain (by which we mean a connected open set) in  $\mathbb R^n=\{(x_1,\ldots,x_n)\}$. 
For scalar valued $u:\Omega\to\re$,  $u\in W^{1, p}_{loc}(\Omega)$, the {\underline{gradient}} of $u$ is the vector of first order weak partial derivatives:   $$\nabla u=(\partial_{x_1}u ,\ldots,\partial_{x_n}u),$$ defined on  $\Omega$. The {\underline{norm}} of a vector $x\in\re^n$ is $|x|=\sqrt{x_1^2+\cdots+x_n^2}$, and in particular the norm of the gradient is defined on the domain of the gradient by
  $$|\nabla u|^2=(\partial_{x_1}u)^2+\cdots+(\partial_{x_n}u)^2.$$

Given $r>0$ and $x_0\in{\mathbb R}^n$, let $B_r(x_0)$ denote the open ball centered at $x_0$ with radius $r$ --- in the special case $x_0=0$, we abbreviate $B_r=B_r(0)$.  For sets $A\subseteq{\mathbb R}^n$, let $\partial A$ denote the set of boundary points of $A$ and let $\overline{A}$ denote the closure of $A$.

In this section, we prove  Theorem \ref{main}, the unique continuation property for vector valued $H^1$ solutions $u:\Omega\to\re^M$, where the inequality \eqref{eq} reads as \begin{equation}\label{va}
     |\nabla u|=\left(\sum_{j=1}^n\sum_{k=1}^M | \partial_{x_j}u_k|^2\right)^{\frac{1}{2}}\le V \left(\sum_{k=1}^M|u_k|^2 \right)^{\frac{1}{2}}= V|u|.
\end{equation}
We first prove the following Hardy-type inequality for $\nabla$. 

\begin{lemma}\label{hardy}
Let $u\in H^1(\mathbb R^n)$ with support outside a neighborhood of $0$. Then for any $\lambda > \frac{n}{2}$,
\begin{equation}\label{ha}
    \int_{\mathbb R^n} \frac{|u(x)|^2}{|x|^{2\lambda}}dv\le \frac{4}{(2\lambda-n)^2}\int_{\mathbb R^n} \frac{|\nabla u(x)|^2}{|x|^{2\lambda-2}}dv. 
\end{equation}
\end{lemma}
\begin{proof}
We first show the inequality when $u\in C_c^\infty(\mathbb R^n\setminus \{0\})$. Let $$F(x):  =\sum_{j=1}^n\frac{|u(x)|^2 x_j}{|x|^{2\lambda}} dx_1\wedge \cdots \wedge  dx_{j-1}  \wedge  dx_{j+1}\wedge \cdots \wedge dx_n.$$ Then $F$ is a smooth $(n-1)$ form with compact support. Note for each $j=1, \ldots, n$,
$$ \sum_{j=1}^n \partial_{x_j}  \left(\frac{ x_j}{|x|^{2\lambda}}\right) =  \sum_{j=1}^n  \left(\frac{ 1}{|x|^{2\lambda}} -\frac{ 2\lambda x_j^2}{|x|^{2\lambda+2}}\right) = \frac{n-2\lambda}{|x|^{2\lambda}}.  $$
Applying Stokes' theorem on $F$, we have
$$ 0 = \int_{\mathbb R^n}d F = \int_{\mathbb R^n} \frac{(n-2\lambda)|u(x)|^2}{|x|^{2\lambda}}dv + \int_{\mathbb R^n}  \frac{2 u(x)  \langle\nabla u(x), x\rangle}{|x|^{2\lambda}}dv.$$
Thus
\begin{equation*}
    \int_{\mathbb R^n} \frac{|u(x)|^2}{|x|^{2\lambda}}dv  = \frac{2}{(2\lambda-n)}\int_{\mathbb R^n}  \frac{u(x) \langle\nabla u(x), x\rangle}{|x|^{2\lambda}}dv.
\end{equation*}
By the Cauchy-Schwarz inequality, one further gets
\begin{equation*}\begin{split}
     \int_{\mathbb R^n} \frac{|u(x)|^2}{|x|^{2\lambda}}dv   \le& \frac{2}{(2\lambda-n)}\int_{\mathbb R^n}  \frac{|u(x)||\nabla u(x)|}{|x|^{2\lambda-1}}dv\\
     \le&\frac{2}{(2\lambda-n)}\left(\int_{\mathbb R^n} \frac{|u(x)|^2 }{|x|^{2\lambda}}dv\right)^{\frac{1}{2}}\left(\int_{\mathbb R^n}  \frac{|\nabla  u(x) |^2}{|x|^{2\lambda-2}}dv\right)^\frac{1}{2}.
\end{split}
   \end{equation*}
Dividing both sides  by $\left(\int_{\mathbb R^n} \frac{|u(x)|^2 }{|x|^{2\lambda}}dv\right)^{\frac{1}{2}} $ and then squaring both sides, we obtain
 \eqref{ha} for $u\in C_c^\infty(\mathbb R^n\setminus \{0\})$.
 
For general $u\in H^1(\mathbb R^n)$ with support, say,  away from $B_r, r>0$, we use the standard density argument. In detail, let $u^{(j)} \in C_c^\infty(\mathbb R^n\setminus B_r)\rightarrow u$ in $H^1$ norm. Then
\begin{equation*}
    \begin{split}
    \left(\int_{\mathbb R^n} \frac{|u(x)|^2}{|x|^{2\lambda}}dv\right)^\frac{1}{2}\le& \left(\int_{\mathbb R^n\setminus B_r}\frac{|u(x)-u^{(j)} (x)|^2}{|x|^{2\lambda}}dv\right)^\frac{1}{2}+ \left(\int_{\mathbb R^n} \frac{|u^{(j)} (x)|^2}{|x|^{2\lambda}}dv\right)^\frac{1}{2}\\
    \le& \frac{1}{r^{\lambda}}\left(\int_{\mathbb R^n}|u(x)-u^{(j)} (x)|^2dv\right)^\frac{1}{2}+  \frac{2}{(2\lambda-n)}\left(\int_{\mathbb R^n} \frac{|\nabla u^{(j)} (x)|^2}{|x|^{2\lambda-2}}dv\right)^\frac{1}{2}.
    \end{split}
\end{equation*}
Here we used \eqref{ha} for $u^{(j)}\in C_c^\infty(\mathbb R^n\setminus \{0\})$. Thus
\begin{eqnarray*}
  &&\left(\int_{\mathbb R^n} \frac{|u(x)|^2}{|x|^{2\lambda}}dv\right)^\frac{1}{2} \\
  &\le&  \frac{1}{r^{\lambda}}\left(\int_{\mathbb R^n}|u(x)-u^{(j)} (x)|^2dv\right)^\frac{1}{2}+  \frac{2}{(2\lambda-n)}\left(\int_{\mathbb R^n\setminus B_r} \frac{|\nabla u^{(j)} (x) -\nabla u(x)|^2}{|x|^{2\lambda-2}}dv\right)^\frac{1}{2} \\
  &&\ +  \frac{2}{(2\lambda-n)}\left(\int_{\mathbb R^n}  \frac{|\nabla  u(x) |^2}{|x|^{2\lambda-2}}dv\right)^\frac{1}{2}\\
    &\le & \frac{1}{r^{\lambda}}\left(\int_{\mathbb R^n}|u(x)-u^{(j)} (x)|^2dv\right)^\frac{1}{2}+ \frac{2}{(2\lambda-n)r^{\lambda-1}} \left(\int_{\mathbb R^n} |\nabla u^{(j)} (x) -\nabla u(x)|^2dv\right)^\frac{1}{2} \\
    &&\ +  \frac{2}{(2\lambda-n)}\left(\int_{\mathbb R^n}  \frac{|\nabla  u(x) |^2}{|x|^{2\lambda-2}}dv\right)^\frac{1}{2}\\
   & \le&  \left(\frac{1}{r^{\lambda}}+  \frac{2}{(2\lambda-n)r^{\lambda-1}}\right)\|u-u^{(j)} \|_{H^1(\mathbb R^n)}+  \frac{2}{(2\lambda-n)}\left(\int_{\mathbb R^n}  \frac{|\nabla  u(x) |^2}{|x|^{2\lambda-2}}dv\right)^\frac{1}{2}.
\end{eqnarray*}
Letting $j\rightarrow \infty$, we   have the desired inequality \eqref{ha}.
\end{proof}

\begin{lemma}\label{in}
Let $u\in H^1(\mathbb R^n)$ with support outside a neighborhood of 0.  Then there exists a constant $C_0>0$ such that for any $\lambda >>\frac{n}{2}$,
$$\int_{\mathbb R^n} \left|\nabla \left(\frac{ u(x)}{|x|^{\lambda-1}}\right)\right|^2dv\le C_0 \int_{\mathbb R^n} \frac{|\nabla u(x)|^2 }{|x|^{2\lambda-2}}dv.$$
\end{lemma}
\begin{proof}
 Since $|\nabla |x|| =1$, we have
\begin{equation*}
    \begin{split}
 \int_{\mathbb R^n} \left|\nabla \left(\frac{ u(x)}{|x|^{\lambda-1}}\right)\right|^2dv\le& 2\int_{\mathbb R^n} \frac{|\nabla u(x)|^2 }{|x|^{2\lambda-2}}dv + 2(\lambda-1)^2\int_{\mathbb R^n} \frac{| u(x)|^2 }{|x|^{2\lambda}}dv\\
\le& 2\int_{\mathbb R^n} \frac{|\nabla u(x)|^2 }{|x|^{2\lambda-2}}dv + \frac{8(\lambda-1)^2}{(2\lambda-n)^2}\int_{\mathbb R^n} \frac{|\nabla u(x)|^2 }{|x|^{2\lambda-2}}dv\\
=&\left(2+  \frac{8(\lambda-1)^2}{(2\lambda-n)^2}    \right)\int_{\mathbb R^n} \frac{|\nabla u(x)|^2 }{|x|^{2\lambda-2}}dv.
    \end{split}
\end{equation*}
Here in the second inequality  we used  Lemma \ref{hardy}. The lemma thus follows from the fact that $\displaystyle{\lim_{\lambda\rightarrow \infty} \frac{8(\lambda-1)^2}{(2\lambda-n)^2} = 2}$.
\end{proof}

Throughout the rest of the paper, we occasionally  use the notation $a\lesssim b$ 
 for two quantities $a$ and $b$,  to mean that there exists a universal constant  $C$ (dependent only possibly on $n$)  such that $a \le Cb$. To prove Theorem \ref{main} in the case when $n=2$, we will use the following unique continuation property established in \cite{PZ} for $\bar\partial$. Note that identifying $z\in \mathbb C$ with $(x_1, x_2)\in \mathbb R^2$, then  for a function $u$ on $\Omega$, $\bar\partial_z u= \frac{1}{2}\left( \partial_{x_1} u +  i\partial_{x_2}u\right)$. It would be interesting to have a real-variable approach for this case, but we currently do not.

\begin{proposition}\cite{PZ}\label{pz}
Let $\Omega$ be a  domain in $\mathbb C$. Suppose $u=(u_1, \ldots, u_N): \Omega\rightarrow \mathbb C^N$ with $u\in H^1_{loc}(\Omega)$   and satisfies $ |\bar\partial u|\le V|u|$ a.e.\ on $\Omega$
 for some $V\in L_{loc}^2(\Omega)$.  If $u$ vanishes to infinite order at $z_0\in \Omega$, then $u$ vanishes identically.   \boxx
\end{proposition}

\begin{proof}[Proof of Theorem \ref{main}:]
The $n=2$ case follows from Proposition \ref{pz} and the trivial fact that  $ |\bar \partial u|\lesssim  |\nabla u| $. When $n\ge3$, without loss of generality  assume $x_0=0$. Fix $r\in (0, 1)$  so that \begin{equation}\label{C}
   \left( \int_{B_{2r}}|V(x)|^ndv \right)^{\frac{2}{n}}<\frac{1}{2C_0C_1^2},
\end{equation}
where $C_0$ is the constant in Lemma \ref{in}, and $C_1$ is the constant in the Gagliardo-Nirenberg-Sobolev inequality:
$$ \|f\|_{L^{\frac{2n}{n-2} }(\mathbb R^n) }\le C_1 \|\nabla f\|_{L^2(\mathbb R^n)}, \ \ \text{for all}\ \ f\in H^1(\mathbb R^n). $$
We shall show that $u=0$ in $B_{\frac{r}{2}}$. Thus, applying a standard propagation argument we obtain $u\equiv 0$ on $\Omega$.

Choose  $\eta\in C_c^\infty(\mathbb R^n)$   such that  $0\le \eta\le 1$, $\eta =1$ on $B_r$,  $\eta =0$  outside $B_{2r}$, and $|\nabla \eta|\le \frac{2}{r}$ on $B_{2r}\setminus  B_r$. Let $\psi\in  C^\infty(\mathbb R^n)$ be  such that $  0\le \psi\le 1$, $\psi =0$ in $B_1$,  $\psi =1$ outside $B_2$, and $|\nabla \psi|\le 2$ on $B_{2}\setminus {B_1}$.  For each $k\ge \frac{4}{r}$ (then $ \frac{2}{k}\le\frac{r}{2}$),  let $\psi_k(x) = \psi(kx), x\in\mathbb R^n$.  Defining $u^{(k)}=\psi_k \eta u$, note that $u^{(k)}\in H^1(\mathbb R^n)$ and is supported inside $B_{2r}\setminus B_{\frac{1}{k}}$.  Then for each $k\ge \frac{4}{r}$ and $\lambda> \frac{n}{2}$,
\begin{eqnarray}
         &&\int_{B_{2r}}\frac{|\nabla u^{(k)}(x)|^2}{|x|^{2\lambda-2}}dv\nonumber\\
         &\lesssim & \int_{B_{2r}}\frac{|\psi_k(x) \eta(x) |^2 | \nabla u(x)|^2}{|x|^{2\lambda-2}}dv  + \int_{B_r}\frac{|\nabla \psi_k(x)|^2 |  u(x)|^2}{|x|^{2\lambda-2}}dv + \int_{B_{2r}\setminus B_r}\frac{|\nabla \eta(x)  |^2 |  u(x)|^2}{|x|^{2\lambda-2}}dv \nonumber\\
        &\le & \int_{B_{2r}}\frac{|V(x)|^2 |\psi_k(x) \eta(x)  u(x)|^2}{|x|^{2\lambda-2}}dv  + \int_{B_r}\frac{|\nabla \psi_k(x)|^2 |  u(x)|^2}{|x|^{2\lambda-2}}dv + \int_{B_{2r}\setminus B_r}\frac{|\nabla \eta(x)  |^2 |  u(x)|^2}{|x|^{2\lambda-2}}dv\nonumber\\
        &\le & \left( \int_{B_{2r}}|V(x)|^ndv \right)^{\frac{2}{n}}\left(\int_{\mathbb R^n} \left(\frac{|u^{(k)}(x)|}{|x|^{\lambda-1}}\right)^{\frac{2n}{n-2}}dv  \right)^{\frac{n-2}{n}}+ \int_{B_r}\frac{|\nabla \psi_k(x)|^2 |  u(x)|^2}{|x|^{2\lambda-2}}dv\nonumber\\
        &&\ + \int_{B_{2r}\setminus B_r}\frac{|\nabla \eta(x)  |^2 |  u(x)|^2}{|x|^{2\lambda-2}}dv.\label{eq2.4}
\end{eqnarray}
Here we have used H\"older's inequality in (\ref{eq2.4}). Since  $\frac{|u^{(k)}(x)|}{|x|^{\lambda-1}}\in H^1(\mathbb R^n), n\ge 3$, making use of the Gagliardo-Nirenberg-Sobolev inequality  and Lemma \ref{in}, we get
\begin{equation*}
    \begin{split}
    \left(\int_{\mathbb R^n} \left(\frac{|u^{(k)}(x)|}{|x|^{\lambda-1}}\right)^{\frac{2n}{n-2}}dv \right)^{\frac{n-2}{n}}\le& C_1^2\int_{\mathbb R^n} \left|\nabla\left(\frac{|u^{(k)}(x)|}{|x|^{\lambda-1}}\right)\right|^2dv       \le C_0 C_1^2\int_{B_{2r}} \frac{|\nabla u^{(k)}(x)|^2 }{|x|^{2\lambda-2}}dv.
    \end{split}
\end{equation*}
This combined  with \eqref{eq2.4} and \eqref{C}   for each $k\ge \frac{4}{r}$ and $\lambda> \frac{n}{2}$ leads to
\begin{equation}\label{so}
    \int_{B_{2r}}\frac{|\nabla u^{(k)}(x)|^2}{|x|^{2\lambda-2}}dv\le 2\int_{B_r}\frac{|\nabla \psi_k(x)|^2 |  u(x)|^2}{|x|^{2\lambda-2}}dv+ 2\int_{B_{2r}\setminus B_r}\frac{|\nabla \eta(x)  |^2 |  u(x)|^2}{|x|^{2\lambda-2}}dv.
\end{equation}

Now suppose toward a contradiction that $\nabla u \not\equiv 0$ on $B_\frac{r}{2}$. Then
there   exists  $k_1>0$  such that \begin{equation}\label{co}
    M_1=   \int_{B_{\frac{r}{2}}\setminus B_{\frac{2}{k_1}}} |\nabla u(x)|^2 dv > 0.
\end{equation}
Consequently for each fixed $\lambda > \frac{n}{2}$, \begin{equation*}
  M_\lambda=   \int_{B_{\frac{r}{2}}\setminus B_{\frac{2}{k_1}}}\frac{|\nabla u(x)|^2}{|x|^{2\lambda-2}}dv > 0.
\end{equation*}
Noting that $\nabla u^{(k)} =  \nabla u$ on $B_\frac{r}{2}\setminus B_{\frac{2}{k_1}} $ for any $k\ge k_1$ by construction of $u^{(k)}$, we further have for any $k\ge k_1$,
\begin{equation}\label{co3}
    \int_{B_\frac{r}{2}}|\nabla u^{(k)}(x)|^2 dv\ge \int_{B_{\frac{r}{2}}\setminus B_{\frac{2}{k_1}}}|\nabla u(x)|^2dv= M_1 >0,
\end{equation}
and
\begin{equation}\label{co2}
    \int_{B_\frac{r}{2}}\frac{|\nabla u^{(k)}(x)|^2}{|x|^{2\lambda-2}}dv\ge \int_{B_{\frac{r}{2}}\setminus B_{\frac{2}{k_1}}}\frac{|\nabla u(x)|^2}{|x|^{2\lambda-2}}dv= M_\lambda >0.
\end{equation}

On the other hand,  by flatness of $u$ at $0$, \begin{equation}\label{fl}
    \int_{B_r}\frac{|\nabla \psi_k(x)|^2 |  u(x)|^2}{|x|^{2\lambda-2}}dv =\int_{B_{\frac{2}{k}}\setminus B_{\frac{1}{k}}}\frac{|\nabla \psi_k(x)|^2 |  u(x)|^2}{|x|^{2\lambda-2}}dv \le  4k^{2\lambda}\int_{B_{\frac{2}{k}}}| u(x)|^2dv \rightarrow 0
\end{equation} as $k\rightarrow \infty$. In particular, by \eqref{co2} one can get some $k_\lambda>k_1$  such that
$$ \int_{B_r}\frac{|\nabla \psi_{k_\lambda}(x)|^2 |  u(x)|^2}{|x|^{2\lambda-2}}dv\le \frac{M_\lambda}{4} \le \frac{1}{4} \int_{B_{r}}\frac{|\nabla u^{(k_\lambda)}(x)|^2}{|x|^{2\lambda-2}}dv.$$
Thus \eqref{so} with $k=k_\lambda$   becomes
\begin{equation}\label{hh}
    \begin{split}
        &\int_{B_{2r}}\frac{|\nabla u^{(k_\lambda)}(x)|^2}{|x|^{2\lambda-2}}dv\le    4\int_{B_{2r}\setminus B_r}\frac{|\nabla \eta(x)  |^2 |  u(x)|^2}{|x|^{2\lambda-2}}dv.
    \end{split}
\end{equation}
Since $$ \int_{B_{2r}}\frac{|\nabla u^{(k_\lambda)}(x)|^2}{|x|^{2\lambda-2}}dv\ge \int_{B_{\frac{r}{2}}}\frac{|\nabla u^{(k_\lambda)}(x)|^2}{|x|^{2\lambda-2}}dv\ge \left(\frac{2}{r}\right)^{2\lambda-2}\int_{B_{\frac{r}{2}}}|\nabla u^{(k_\lambda)}(x)|^2dv $$
and $$\int_{B_{2r}\setminus B_r}\frac{|\nabla \eta(x)  |^2 |  u(x)|^2}{|x|^{2\lambda-2}}dv \le \frac{1}{r^{2\lambda-2}}\int_{B_{2r}\setminus B_r}|\nabla \eta(x)  |^2 |  u(x)|^2dv,$$
we obtain from \eqref{hh} that
$$  2^{2\lambda-4}\int_{B_{\frac{r}{2}}}|\nabla u^{(k_\lambda)}(x)|^2dv\le \int_{B_{2r}\setminus B_r}|\nabla \eta(x)  |^2 |  u(x)|^2dv.  $$
Letting $\lambda\rightarrow \infty$ and making use of the fact that $u\in H^1_{loc}(\Omega)$, we see that
$$ \int_{B_{\frac{r}{2}}}|\nabla u(x)|^2dv =0.  $$
But this 
 contradicts \eqref{co}!  We thus have $\nabla u \equiv 0$ on $B_\frac{r}{2}$. By flatness of  $u$ at $0$,  $u$ must be zero on $B_\frac{r}{2}$. 
\end{proof}

\begin{example}\label{ex1.2} 
Given  $0<\varepsilon<\frac{n-1}{n}, n\ge 2$, consider the differential equation $$ |\nabla u| = V|u| \ \ \text{on}\ \  B_{\frac{1}{2}},$$ where $$V=  \frac{\varepsilon(-\log |x|)^{\varepsilon-1}}{|x|}\ \ \text{on}\ \ B_{\frac{1}{2}}.$$
It is straightforward  to verify that $V\in  L^n(B_{\frac{1}{2}})$. As a consequence of Theorem \ref{main}, every nonconstant $H^1 $ solution must vanish to finite order in the $L^2$ sense at each point in $B_{\frac{1}{2}} $. 

 On the other hand, the function $$u_0(x) =e^{-{(-\log |x|)^\varepsilon}}$$ (extended to the origin by $u_0(0)=0$) is continuous on $B_{\frac{1}{2}} $, and smooth on $B_{\frac{1}{2}}\setminus\{0\}$.  Moreover,  $u_0\in H^1( B_{\frac{1}{2}}) $ and is a solution of $|\nabla u|=V u$ a.e.\ on $B_{\frac{1}{2}}$. Note that there is no contradiction with Theorem \ref{main} since $u_0$ vanishes to finite order in the $L^2$ sense everywhere in $B_{\frac{1}{2}} $.
\end{example}

When $V\in L^p$, $p<n$,  the   unique continuation property fails in general as seen below.

\begin{example}\label{ex}
For each $1\le p<n$, and  $0<\epsilon < \frac{n-p}{p}$ (so that $(\epsilon+1)p<n$), $$u(x) = e^{-\frac{1}{|x|^\epsilon}} $$ (extended to the origin by $u(0)=0$) is a smooth function on $B_1$ and  vanishes to infinite order at $0$. Moreover, the function $u$ satisfies $ |\nabla  u|\le V|u|$  on $B_1$ with
$$V = \frac{\epsilon}{|x|^{\epsilon+1}}\in L^p(B_1).$$
\end{example}

On the other hand, the following theorem states that for some special potentials in the form of multiples of $\frac{1}{|x|}$,   the   unique continuation property can still hold. Note  that $\frac{1}{|x|}\notin L^n_{loc} $.

\begin{theorem}\label{main2}
Let $\Omega$ be a  domain in $\mathbb R^n, n\ge 1$. Suppose $u=(u_1, \ldots, u_M): \Omega\rightarrow \mathbb R^M$ with $u\in H^1_{loc}(\Omega)$  and satisfies $ |\nabla u|\le \frac{C}{|x|}|u|$ a.e.\ for some constant $C>0$.  If $u$ vanishes to infinite order at some $x_0\in \Omega$ in the $L^2$ sense, then $u$ vanishes identically.
\end{theorem}
\begin{proof}
Assume $x_0=0$ and consider $u^{(k)} = \psi_k\eta u$, where  $\psi_k$ and $\eta$ are defined as in the proof of Theorem \ref{main}.
Then by Lemma \ref{hardy},
\begin{eqnarray*}
     &&\int_{B_{2r}}\frac{|u^{(k)}(x)|^2}{|x|^{2\lambda}}dv\\
     &\le& \frac{4}{(2\lambda-n)^2}\int_{B_{2r}}\frac{|\nabla u^{(k)}(x)|^2}{|x|^{2\lambda-2}}dv\\
     &  \lesssim & \frac{4}{(2\lambda-n)^2}\int_{B_{2r}}\frac{|\psi_k(x) \eta(x) |^2 | \nabla u(x)|^2}{|x|^{2\lambda-2}}dv  + \frac{4}{(2\lambda-n)^2}\int_{B_r}\frac{|\nabla \psi_k(x)|^2 |  u(x)|^2}{|x|^{2\lambda-2}}dv\\  &&\ +\frac{4}{(2\lambda-n)^2} \int_{B_{2r}\setminus B_r}\frac{|\nabla \eta(x)  |^2 |  u(x)|^2}{|x|^{2\lambda-2}}dv \\
    &    \le &  \frac{4C^2}{(2\lambda-n)^2}\int_{B_{2r}}\frac{ |u^{(k)}(x)|^2}{|x|^{2\lambda}}dv  +  \frac{4}{(2\lambda-n)^2}\int_{B_r}\frac{|\nabla \psi_k(x)|^2 |  u(x)|^2}{|x|^{2\lambda-2}}dv\\
     &   &\ +  \frac{4}{(2\lambda-n)^2}\int_{B_{2r}\setminus B_r}\frac{|\nabla \eta(x)  |^2 |  u(x)|^2}{|x|^{2\lambda-2}}dv.
\end{eqnarray*}
Here we used the inequality  $|\nabla u|\le \frac{C}{|x|}|u| $ in the first term of the last inequality. When $ \frac{4C^2}{(2\lambda-n)^2}\le \frac{1}{2} $ (equivalently, when $ \lambda>\frac{n}{2}+\sqrt 2C$),  one can move this first term to the left hand side and get
$$ \int_{B_{2r}}\frac{|u^{(k)}(x)|^2}{|x|^{2\lambda}}dv  \le \frac{8}{(2\lambda-n)^2}\int_{B_r}\frac{|\nabla \psi_k(x)|^2 |  u(x)|^2}{|x|^{2\lambda-2}}dv+\frac{8}{(2\lambda-n)^2}\int_{B_{2r}\setminus B_r}\frac{|\nabla \eta(x)  |^2 |  u(x)|^2}{|x|^{2\lambda-2}}dv. $$

Letting $k\rightarrow \infty$ and making use of the flatness of $u$ with a similar argument as in \eqref{fl}, we obtain
\begin{equation*}
    \int_{B_{2r}}\frac{|u(x)|^2}{|x|^{2\lambda}}dv  \le\frac{16}{(2\lambda-n)^2}\int_{B_{2r}\setminus B_r}\frac{|\nabla \eta(x)  |^2 |  u(x)|^2}{|x|^{2\lambda-2}}dv.
\end{equation*}
  Since
  $$  \int_{B_{2r}}\frac{|u(x)|^2}{|x|^{2\lambda}}dv \ge  \int_{B_{\frac{r}{2}}}\frac{|u(x)|^2}{|x|^{2\lambda}}dv \ge\left(\frac{2}{r}\right)^{2\lambda}\int_{B_{\frac{r}{2}}}| u(x)|^2dv $$
and $$\int_{B_{2r}\setminus B_r}\frac{|\nabla \eta(x)  |^2 |  u(x)|^2}{|x|^{2\lambda-2}}dv \le \frac{1}{r^{2\lambda-2}}\int_{B_{2r}\setminus B_r}|\nabla \eta(x)  |^2 |  u(x)|^2dv,$$
we have
$$\int_{B_{\frac{r}{2}}}| u(x)|^2dv \le \frac{r^2}{(2\lambda-n)^22^{2\lambda-4}}\int_{B_{2r}\setminus B_r} |\nabla \eta(x)  |^2 |  u(x)|^2dv. $$
Letting $\lambda\rightarrow \infty$, we see $u\equiv 0$ on $B_\frac{r}{2}$.
\end{proof}

 Although  the unique continuation property for \eqref{eq} fails for general $ L^p$, $p<n$, potentials as demonstrated in Example \ref{ex}, the following theorem shows that if the potential is in $ L^2 $ then the weak continuation property holds. Recall that weak unique continuation for a differential (in)equality is the property that every solution  that vanishes in an open subset vanishes identically.

\begin{theorem}\label{wucp}
    Let $\Omega$ be a  domain in $\mathbb  R^n, n\ge 2$, and let $V\in L_{loc}^2(\Omega)$. Suppose $u=(u_1, \ldots, u_M): \Omega\rightarrow \mathbb R^M$ with $u\in H^1_{loc}(\Omega)$  and satisfies $ |\nabla  u|\le V|u|$ on $\Omega$.  If $u $ vanishes  in an open subset of $ \Omega$,   then $u$ vanishes identically.
\end{theorem}
\begin{proof}  The $n=2$ case is a direct consequence of Theorem \ref{main}, since  the (strong) unique continuation implies the weak unique continuation property. We shall show below when $n = 3$,  for any two domains $D_1, D_2$ in $\mathbb R^2$ with $D_1\subseteq D_2$, and $s>0$, if $u$ satisfies \eqref{eq} on the product domain $D_2 \times (-s, s) $ and $ u= 0$ on $D_1\times (-s, s)$, then $u =0 $ on $D_2\times (-s, s)$. If so, then $u\equiv 0$ with a standard propagation argument. The proof for $n\ge 3$ cases follows from an induction.

Since $V\in L_{loc}^2( D_2  \times (-s, s)  )$, by Fubini's theorem, for almost every $x_3\in (-s, s) $, $V(\cdot, x_3)\in L^2_{loc}(D_2)$, and similarly  $u(\cdot, x_3)\in H^1_{loc}(D_2)$. Restricting \eqref{va} at each such $x_3= c_3\in (-s, s)$, we have  $v= u(\cdot, c_3)$ satisfies
\begin{equation*}
    |\nabla v| =\left(\sum_{k=1}^M|\partial_{x_1}u_k(\cdot, c_3)|^2+ |\partial_{x_2}u_k(\cdot, c_3)|^2\right)^{\frac{1}{2}}\le |\nabla u(\cdot, c_3)| \le V(\cdot, c_3)|u(\cdot, c_3)| =V(\cdot, c_3)|v| 
\end{equation*}
 on $D_2$ and  $v= 0 $ on $D_1$. Applying the $n=2$ case we have $v= 0$ on $D_2$. Thus $u = 0$ on $D_2\times (-s, s)$.
\end{proof}

\section{Uniqueness for Lipschitz functions}\label{sec3}

In this section, we focus on locally Lipschitz functions whose definition is given below.

\begin{definition}\label{def2.1}
  A function $u:\Omega\to\re$ is said to be {\underline{locally Lipschitz}} on $\Omega$ means that for any point $p\in\Omega$,  there is some neighborhood $p\in U_p\subseteq\Omega$ and some constant $C_p$ so that for all $x,y\in U_p$, $|u(y)-u(x)|<C_p|y-x|$. The function $u$ is {\underline{Lipschitz}} on $\Omega$ means that there exists a constant $C$ such  that for all $x,y\in \Omega$, $|u(y)-u(x)|<C|y-x|$.
\end{definition}
According to Rademacher's Theorem, if $u$ is locally Lipschitz on $\Omega$, then $\nabla u$ is defined a.e.\ on $\Omega$.  See, for instance, \cite[pp. 296]{Ev}.    Moreover,

\begin{proposition}\cite[pp. 294]{Ev}\label{prop2.4} Let $\Omega$ be a domain in $\mathbb R^n.$ Then
 $u$ is locally Lipschitz on $\Omega$ if and only if $u\in W^{1, \infty}_{loc}(\Omega)$.  \boxx
 \end{proposition}
Following the convention of \cite{Ev}, even for $u\in W^{1, \infty}_{loc}(\Omega)$ defined a.e.\ in $\Omega$ or with some measure zero set of discontinuities, there is a unique continuous function agreeing with $u$ a.e., which we will also denote $u$.


%


%
To prove Theorem \ref{thm1.1}, we begin with a uniqueness property of Lipschitz functions in one real variable on an interval, making use of  the following fundamental theorem of calculus for Lipschitz functions. 

\begin{proposition}\cite[Theorem 7.20, Fundamental Theorem of Calculus]{rudin} \label{prop2.5}
    If $u:[0,1]\to\re$ is Lipschitz on $[0,1]$, then for any $0\le a<b\le1$, $$u(b)-u(a)=\int_a^bu^\prime(t)dt.$$ \boxx
\end{proposition}
\begin{lemma}\label{lem3.1}
Let $\varphi:[0,1]\to\re$ be Lipschitz on $[0,1]$, with $\varphi(0)=0$. If there exist $p\geq 1$ and a non-negative function $\lambda:[0,1]\to\re$ with $\lambda\in L^p([0,1])$ such that for a.e.\ $x\in(0,1)$,
\begin{equation}\label{h}
|\varphi^\prime(x)|\leq \lambda(x)\,|\varphi(x)|\, x^{\frac{1-p}{p}},
\end{equation}
then $\varphi\equiv 0$ in $[0,1]$.
\end{lemma}
\begin{proof}
  We note first that we can assume without loss of generality that $\lambda$ is non-vanishing. Indeed, if that is not the case, then we can just replace $\lambda$ with $1+\lambda\in L^p([0,1])$ and (\ref{h}) still holds.

Let $\delta=\sup\{d\in[0,1]\,\vert\, \varphi\equiv 0 \text{ in }[0,d]\}$. By continuity, $\varphi(\delta)=0$, and by Proposition \ref{prop2.5} (which uses the Lipschitz hypothesis), for all $x\in(0,1]$,
\begin{equation}\label{eq3.2}
  |\varphi(x)|=|\varphi(x)-\varphi(\delta)|=\left|\int_\delta^x\varphi^\prime(t)dt\right|\le\int_\delta^x\left|\varphi^\prime(t)\right|dt.
\end{equation}
The existence of the RHS integral is from Proposition \ref{prop2.4}, with $L^\infty([\delta,x])\subseteq L^p([\delta,x])\subseteq L^1([\delta,x])$.

For $p>1$, let $q$ be the conjugate exponent so that $\frac1p+\frac1q=1$.  By H\"older's inequality, we have
\begin{equation}\label{eq3.3}
\int_{\delta}^{x}|\varphi^\prime(t)|\, dt\leq\bigg(\int_{\delta}^x|\varphi^\prime(t)|^p\, dt\bigg)^{\frac{1}{p}}\bigg(\int_{\delta}^x 1^q\, dt\bigg)^{\frac{1}{q}}\le\bigg(\int_{\delta}^x|\varphi^\prime(t)|^p\, dt\bigg)^{\frac{1}{p}}x^{\frac{p-1}p}.
\end{equation}
It follows from (\ref{eq3.2}) and (\ref{eq3.3}) that for $p\ge1$,
\begin{equation}\label{1}
|\varphi(x)|^p\leq x^{p-1}\int_{\delta}^x|\varphi^\prime(t)|^p\, dt.
\end{equation}
We multiply both sides of \eqref{1} by the function $x^{1-p}\lambda^p(x)$, to obtain:
\begin{equation}\label{2}
\lambda^p(x)\,|\varphi(x)|^px^{1-p}\leq\lambda^p(x)\int_{\delta}^x|\varphi^\prime(t)|^p\, dt.
\end{equation}
Suppose toward a contradiction that $\delta<1$, and let $s\in (\delta,1)$. Integrating in the variable $x$ on both sides of \eqref{2} gives
\begin{equation}\label{3}
\int_{\delta}^s\lambda^p(x)\,|\varphi(x)|^p\, x^{1-p}\, dx\leq \int_{\delta}^s\bigg{(}\lambda^p(x)\int_{\delta}^x|\varphi^\prime(t)|^p\, dt\bigg{)}\,dx.
\end{equation}
Note that the Lipschitz property of $\varphi$ on $[0,1]$ and $\varphi(0)=0$ imply there is some constant $C$ so that $|\varphi(x)|\le C|x|$, so $|\varphi(x)|^p\, x^{1-p}$ is continuous and bounded as a function of $x$ on $(0,1]$.  Then the hypothesis $\lambda\in L^p([0,1])$ applies, so that both the LHS and RHS integrals in (\ref{3}) exist.  The inequality \eqref{3} then implies, first using $x\leq s$, and then the hypothesis (\ref{h}):
\begin{eqnarray}
\int_{\delta}^s\lambda^p(x)\,|\varphi(x)|^p\, x^{1-p}\, dx&\leq& \bigg(\int_{\delta}^s \lambda^p(x)\, dx\bigg)\bigg(\int_{\delta}^s|\varphi^\prime(x)|^p\, dx\bigg)\nonumber\\
  &\leq&  \bigg(\int_{\delta}^s \lambda^p(x)\, dx\bigg)\,\bigg(\int_{\delta}^s\lambda^p(x)\,|\varphi(x)|^p\, x^{1-p}\, dx\bigg).\label{7}
\end{eqnarray}
 By the construction of $\delta$ as the supremum of a set where $\varphi(x)\equiv0$, we can find a sequence of points $s_j\in (\delta,1)$ so that $s_j$ is decreasing, $\displaystyle{\lim_{j\to\infty}s_j= \delta}$, and $\varphi(s_j)\ne0$.   By the continuity of $|\varphi(x)|^p\, x^{1-p}$ and the property that $\lambda^p\ge1$, the integrand $\lambda^p(x)\,|\varphi(x)|^p\, x^{1-p}$ is strictly positive in some neighborhood of $s_j$.  So, for all $j=1,2,3,\ldots$,
\begin{equation*}\label{8}
\int_{\delta}^{s_j}\lambda^p(x)\,|\varphi(x)|^p\, x^{1-p}\, dx> 0.
\end{equation*}
The inequality \eqref{7} then yields, for all $j$,
\begin{equation}\label{final}
1\leq\int_{\delta}^{s_j}\lambda^p(x)\,dx.
\end{equation}
Since $\lambda\in L^p(0,1)$, letting $s_j\to \delta$ in \eqref{final} leads to a contradiction.
\end{proof}

Recalling Rademacher's theorem that a Lipschitz function is differentiable almost everywhere, the following simple, but useful, Lemma gives a set of points where the square of a Lipschitz function is known to be differentiable.

\begin{lemma}\label{lem3.2}
 Let $u$ be a locally Lipschitz function on an open set $\Omega\subseteq\re^n$. Then $g=u^2$   is also locally Lipschitz on $\Omega$. Moreover, $g$ is differentiable wherever $u$ vanishes and in fact $\nabla g(x)=0$ there.
\end{lemma}
\begin{proof}
The locally Lipschitz property of $g$ follows from the well-known fact that the product of locally Lipschitz functions is locally Lipschitz; this is easily checked as an elementary consequence of Definition \ref{def2.1}.
The second claim is also elementary; let $x_0\in\Omega$ be such that $u(x_0)=0$. The properties that $g=u^2$ is differentiable at $x_0$ and $\nabla g(x_0)=0$ follow from the definition of differentiability,
  $$\lim_{x\to x_0}\frac{g(x)-g(x_0)-0\cdot(x-x_0)}{|x-x_0|}=\lim_{x\to x_0}\frac{u^2(x)}{|x-x_0|}=0,$$
  where we have used the Lipschitz property of $u(x)=u(x)-u(x_0)=O(|x-x_0|)$ in a neighborhood of $x_0$.
\end{proof}


\begin{lemma}\label{lem3.3}
For $n\ge 2$, let $A$ be a set of measure zero in the unit ball in $\mathbb{R}^n$. Then for almost all points $\omega$ in the unit sphere $  S^{n-1}$, the set of intersection of $A$ with the radius segment $\{r\omega: 0\leq r\leq 1\}$ is of measure zero in the line measure.
\end{lemma}
\begin{proof}
  Let $|K|$ denote the $d$-dimensional measure of a measurable set $K\subseteq\re^d$, and let $\chi_A:\re^n\to\re$ be the characteristic function of the set $A$. We have $$0=|A|=\int_{|x|<1} \chi_A(x)dv=\int_{S^{n-1}}\int_0^1\chi_A(r\omega)r^{n-1}drd\omega.$$ By Fubini's theorem, we conclude that for a.e.\ $\omega\in S^{n-1}$, $\int_0^1\chi_A(r\omega)r^{n-1}dr=0$, which is the desired result: $|A\cap\{r\omega\}|=0$.
\end{proof}


 Given a locally Lipschitz function $u$ on $\Omega$, denote by $Z_u$ be the zero set of $u$ in $\Omega$, that is, $Z_u=\{x\in \Omega\,\vert\, u(x)=0\}$. Theorem \ref{thm1.1} will be a consequence of the following general result concerning Lipschitz functions.
 
\begin{theorem}\label{thm3.4}
Let $\Omega$ be a domain in $\mathbb R^n, n\ge 2$, and $u$ be a locally Lipschitz function on $\Omega$.   If the zero set ${Z_u}$ of $u$ is neither $\emptyset $ nor $\Omega$, then
\begin{equation}\label{hyp}
\int_{\Omega\setminus {Z_u}}\bigg{|}\frac{\nabla u(x)}{u(x)}\bigg{|}^n\, dv=\infty.
\end{equation}
\end{theorem}
\begin{proof} First we make the observation that in order to prove the theorem it suffices to prove it for $g=u^2$. In fact, since
$$\int_{\Omega\setminus {Z_u}}\bigg{|}\frac{\nabla g(x)}{g(x)}\bigg{|}^n dv=2^n\int_{\Omega\setminus {Z_u}}\bigg{|}\frac{\nabla u(x)}{u(x)}\bigg{|}^n\, dv, $$
  if the conclusion (\ref{hyp}) is true for $g=u^2$, then it is also true for $u$. Hence we only need to prove (\ref{hyp}) for a function that is the square of a locally Lipschitz function. By Lemma \ref{lem3.2}, the gradient of $g$ is $0$ at every point where $g$ is $0$ (the same set ${Z_u}$ where $u$ is $0$), and $g$ also satisfies the locally Lipschitz assumption. For the rest of the proof we assume (by replacing $u$ with $g$) that $\nabla u(x)=0$ wherever $u(x)=0$.
Let
\begin{equation}
V(x)=\begin{cases}
\Big{|}\frac{\nabla u(x)}{u(x)}\Big{|}\quad x\in \Omega\setminus {Z_u}\\
0\quad \quad \quad\,\,\,x\in {Z_u}.
\end{cases}
\end{equation}
Note that $V$ is a measurable function in $\Omega$.  The zero set ${Z_u}$ is closed in $\Omega$, and by the assumptions that ${Z_u}\neq\emptyset$, ${Z_u}\neq\Omega$, and $\Omega$ is connected, there is some boundary point $x_0\in\partial {Z_u}\subseteq {Z_u}\subseteq\Omega$, and a ball $B_{r_0}(x_0)$  such that the closure   $\overline{B_{r_0}(x_0)}\subseteq\Omega$,  
 and $u$ is Lipschitz on   $\overline{B_{r_0}(x_0)}$.   We can assume, after a translation and scaling, that $x_0$ is the origin and the radius $r_0$ is equal to $1$.  Because $B_1\setminus {Z_u}$ is open and non-empty, there is some $B_{r_1}(x_1)\subseteq B_1$ where $u$ is nonvanishing.

Now suppose toward a contradiction that \eqref{hyp} is false:\begin{equation}
\int_{\Omega} V^n(x) \, dv= \int_{\Omega\setminus {Z_u}}\bigg{|}\frac{\nabla u(x)}{u(x)}\bigg{|}^n\, dv< \infty,
\end{equation}
and therefore $V\in L^n(\Omega)$.  Hence in polar coordinates,
\begin{equation}\label{fin}
\int_{B_1}V^n(x)\, dv=\int_{S^{n-1}}\int_{0}^1 V^n(r\omega)r^{n-1}\, drd\omega.
\end{equation}

Since the integral \eqref{fin} is finite, Fubini's theorem implies that for a.e.\ $\omega\in S^{n-1}$ we have
\begin{equation}\label{a}
\int_0^1 V^n(r\omega)r^{n-1}\, dr<\infty.
\end{equation}
From  Rademacher's theorem, let $A\subseteq\Omega$ be the set of measure zero where $\nabla u(x)$ does not exist at $x$.
Choose $\omega_0\in S^{n-1}$  such that \eqref{a} holds, that is, $V(r\omega_0)r^{\frac{n-1}{n}}\in L^n([0,1])$ and at the same time, by Lemma \ref{lem3.3} the same $\omega_0\in S^{n-1}$ can be chosen such that $\nabla u(x)$ exists a.e.\ on the radius segment $\{r\omega_0\}$.
Define $\varphi$, for $t\in[0,1]$, by \[\varphi(t)=u(t\omega_0).\]
It is evident that $\varphi(t)$ is Lipschitz on $[0,1]$ from the Lipschitz property of $u$.
Then applying the chain rule at points $t$ such that $u$ is differentiable at $t\omega_0$, we have
$$\varphi^\prime(t)=\nabla u(t\omega_0)\cdot\omega_0,$$
which implies, for a.e.\ $t\in[0,1]$,
\begin{equation}
|\varphi^\prime(t)|\leq |\nabla u(t\omega_0)|.
\end{equation}
By the definition of $V$, we have
\begin{equation}
|\varphi^\prime(t)|\leq V(t\omega_0)|u(t\omega_0)|=V(t\omega_0)|\varphi(t)|\quad \text{for }\,u(t\omega_0)\neq 0.
\end{equation}
However, when $u(t\omega_0)= 0$, we have, by the observation at the beginning of the proof, $\nabla u(r\omega_0)=0$ and therefore $\varphi^\prime(t)=0$. Hence
we have shown that
$$|\varphi^\prime(t)|\leq V(t\omega_0)|\varphi(t)|=V(t\omega_0)t^{\frac{n-1}{n}}|\varphi(t)|t^{-\frac{n-1}{n}}$$ holds for a.e.\ $t\in [0, 1]$. By Lemma \ref{lem3.1}, with $\lambda(t)=V(t\omega_0)t^{\frac{n-1}{n}}$ and $p=n$, $\varphi(t)\equiv 0$.   So $u\equiv0$ on all the radius segments $\{t\omega_0\}$ for a.e.\ $\omega_0\in S^{n-1}$, but this contradicts the fact that $u$ has no  zeros in the ball $B_{r_1}(x_1)$.
\end{proof}

\begin{remark}\label{re}
  The proof of Theorem \ref{thm3.4} actually leads to the following stronger conclusion: under the same assumption as in Theorem \ref{thm3.4},  on every neighborhood $U\subseteq\Omega$ of  a point $a\in \Omega\cap \partial {Z_u}$, one has
   $$\int_{U\setminus {Z_u}}\bigg{|}\frac{\nabla u(x)}{u(x)}\bigg{|}^n\, dv=\infty.$$
\end{remark}
\medskip

\begin{proof}[Proof of Theorem \ref{thm1.1}]
  For the one-dimensional case $n=1$, $\Omega$ is an open interval $(a,b)$ and Lemma \ref{lem3.1} can be used directly.  For any $s\in(x_0,b)$, let $\varphi(x)=u((s-x_0)x+x_0)$ so that $\varphi(0)=u(x_0)=0$, $\varphi(1)=u(s)$, and $\varphi$ is Lipschitz on $[0,1]$.  Lemma \ref{lem3.1} applies to $\varphi$ with $p=1$ and $\lambda(x)=V((s-x_0)x+x_0)\cdot|s-x_0|\in L^1([0,1])$, to show $\varphi(1)=0=u(s)$.  Similarly, $u(t)=0$ for any $a<t<x_0$.

  For $n\ge 2$, suppose $u$ and $V$ satisfy $|\nabla u |\leq V |u |$ a.e.\ on $\Omega$ with $V\in L^n_{loc}(\Omega)$.  Let $B=B_r(x_0)$ be a nonempty open ball such that  $\overline B\subset\Omega$.   Suppose toward a contradiction that $B\not\subseteq {Z_u}$. Then $u$ is locally Lipschitz on $B$, and the zero set  $Z_u\cap B$ of $u$ in $B$ is neither $\emptyset$ (since $x_0\in Z_u\cap B$) nor $B$. However,   $$\int_{B\setminus {Z_u}}\left|\frac{\nabla u(x)}{u(x)}\right|^ndv\le\int_{B\setminus {Z_u}}\left|V(x)\right|^ndv<\infty, $$    contradicting Theorem \ref{thm3.4}.  We can conclude $B\subseteq {Z_u}$. Thus ${Z_u}$ is both open and closed in the connected set $\Omega$ and $u\equiv0$.
\end{proof}

\begin{remark}\label{rem}
    The Lipschitz condition cannot just be dropped in Theorem \ref{thm1.1} when $n\ge 2$.  Indeed, Example \ref{ex1.2} gives a nontrivial function $u_0$ that  is  locally Lipschitz on $B_{\frac{1}{2}}\setminus \{0\}$, continuous on $B_{\frac{1}{2}}$ with $u_0(0)=0$,  and solves  $|\nabla u| = V|u|$ on $ B_{\frac{1}{2}}$ for some    $V\in  L^n(B_{\frac{1}{2}})$. This indicates that the zero set of solutions fails to propagate at a non-Lipschitz point in general.
\end{remark}

On the other hand, the hypothesis of Theorem \ref{thm1.1} can be weakened as in the following Corollary without contradicting Remark \ref{rem} ---  if $u$ is a continuous $k^{th}$ root of a locally Lipschitz function, then the uniqueness still holds.

\begin{corollary}\label{thm4.7}
  Let $\Omega$ be a domain in $\mathbb R^n$ and  $u:\Omega\to\re$ be continuous on $\Omega$ with the zero set $Z_u\subseteq\Omega$.  If there is some integer $k\ge1$ so that $v(x)=(u(x))^k$ is locally Lipschitz on $\Omega$, then $\nabla u$ exists a.e.\ in $\Omega\setminus Z_u$.  Further, if there is some $V \in L^n_{loc}(\Omega)$, so that $\nabla u$ satisfies 
  \begin{equation*}
   |\nabla u|\leq V |u|\mbox{ \ \  \ a.e.\ on $\Omega\setminus Z_u$,}
  \end{equation*}
  and $Z_u\ne \emptyset$, then $u\equiv 0$.
\end{corollary}
\begin{proof}
  On the open set where $u(x)>0$, the partial derivatives of $v$ exist a.e.\ and at each point where $\nabla v$ exists, by the chain rule, the partial derivatives of $u(x)=(v(x))^{1/k}$ exist, with $\nabla u(x)=\frac1k(v(x))^{\frac1k-1}\nabla v(x)$.  Similarly, on the open set where $u(x)<0$, the partial derivatives of $u=-((-1)^kv)^{1/k}$ exist a.e., establishing the first claim.
  
  Consider $g(x)=(v(x))^2=(u(x))^{2k}$.  By Lemma \ref{lem3.2}, $g$ is locally Lipschitz on $\Omega$ and at every point where $v(x)=0$, which is the same set as $Z_u$, the partial derivatives of $g$ exist with $\nabla g=0$.  At every point where $u(x)\ne0$, if $\nabla u$ exists, then $\nabla g$ also exists and is equal to $2ku^{2k-1}\nabla u$.  So, at every point in $\Omega$ except for a set of measure zero contained in $\Omega\setminus Z_u$, $\nabla g$ exists and satisfies:
    \begin{equation*}
        |\nabla g|=2k|u|^{2k-1}|\nabla u|\le 2kV |u|^{2k} = 2kV|g|.
    \end{equation*}
  Applying Theorem \ref{thm1.1} and the assumption   $Z_u\ne \emptyset$, we have $g\equiv 0$, and thus $u\equiv 0$.
\end{proof}

We conclude the section with an application of Theorem \ref{thm1.1}   to a uniqueness problem for a nonlinear system of differential equations. 

\begin{corollary}\label{coo}
    For a domain $\Omega\subseteq\mathbb R^n$, $u:\Omega\to\re$, $x_0\in \Omega$, $y_0\in\mathbb R$, let $f: \mathbb R\rightarrow \mathbb R^n $ be Lipschitz on $\mathbb R$.  Then there exists at most one Lipschitz  solution to   $\nabla u =f(u)$ on $\Omega$ with $u(x_0)=y_0$. 
\end{corollary}
\begin{proof}
    Suppose there exists a pair of Lipschitz solutions  $u_1, u_2$ to $\nabla u =f(u)$ on $\Omega$ with the same initial condition $u_1(x_0)=u_2(x_0)=y_0$. Then $w = u_1-u_2$ is Lipschitz on $\Omega$, $w(x_0)=0$, and $w$ satisfies
    $$|\nabla w| = |\nabla u_1 -\nabla u_2| =  |f(u_1)- f(u_2)|\le C|u_1-u_2| = C|w|\ \ \text{on}\ \ \Omega,  $$ 
    where  $C>0$ is the Lipschitz constant for $f$. By Theorem \ref{thm1.1}, we have $w \equiv 0$.
\end{proof}

\section{Further applications}\label{sec5}

\begin{corollary}\label{cor4.1}
  Given any nonempty closed set $A\subsetneq\mathbb{R}^n$, there exists a smooth function $F:\mathbb{R}^n\setminus A\to\re$ so that for any point $a\in\partial A$, $F$ cannot be extended to an $L^n$ integrable function  over any neighborhood of $a$.
\end{corollary}
\begin{proof}
 First, by a well-known theorem of Whitney (\cite{W34}), there exists a smooth function $h:\mathbb{R}^n\to\re$ whose zero set is exactly $A$.   Its square, $u(x)=(h(x))^2$, is also smooth, has zero set exactly $A$, and satisfies $\nabla u=0$ at every point of $A$ by Lemma \ref{lem3.2}.  The quotient $\frac{|\nabla u|}{|u|}$ is the claimed smooth function $F$ on the open set $\re^n\setminus A$.  If there were some ball $B=B_r(a)$ and a function $V\in L^n(B)$ which agrees with $\frac{|\nabla u|}{|u|}$ on $B\setminus A$, then $u$ and $V$ would satisfy (\ref{eq}) from Theorem \ref{thm1.1} at every point of $B\setminus A$ by construction of $V$, and at every point of $B\cap A$, where $|\nabla u|=0$.  By Theorem \ref{thm1.1}, $u\equiv0$ on $B$, contradicting the assumption that $a$ is a boundary point of the zero set.
\end{proof}

\begin{theorem}\label{t1}
Let $\Omega$ be an open set in $\mathbb{R}^n$, and   $u$ be a   Lipschitz function on $\Omega$. Then   $\log |u(x)-u(a)|\notin W_{loc}^{1,n}(\Omega\setminus Z_{u-u(a)})$ for every $a\in\Omega$. In particular, if  $\log |u(x)|\in W_{loc}^{1,n}(\Omega\setminus Z_{u})$, then $u$ is nowhere zero on  $\Omega$.
If, in addition, $\Omega$ has Lipschitz boundary, then the above results also hold true with $\Omega$ replaced by $\overline{\Omega}$.
\end{theorem}
\begin{proof}
Let $v = u-u(a)$ on $\Omega$. Then $v$ is     Lipschitz   on $\Omega$ with $v(a) =0$. 
If $\log |v|\notin L_{loc}^n(\Omega)$, then we are done. If $\log |v|\in L_{loc}^n(\Omega)$, then one further computes  
 $$ \left|\nabla \log |v|\right| = \frac{\left|\nabla v \right|}{|v|}$$
 wherever $v\ne 0$. If $a$ is an interior point of $Z_v= \{x\in \Omega\,\vert\, v(x)=0\}$, the zero set of $v$, then the theorem is trivially true. If $a\in \partial Z_v\cap \Omega$, then one can apply    Remark \ref{re} to conclude $\nabla \log |v|\notin L_{loc}^n(\Omega\setminus Z_v)$.

In the case when $a\in \partial \Omega$ and $\Omega$ has Lipschitz boundary, if $ v(b) =0$ for  some $b\in \Omega$, then it is reduced to the $a\in \Omega$ case. Thus we assume $a\in \partial \Omega$ and $  v(x)\ne 0$ for all $x\in \Omega$. In particular, this means there exists  a cone $ S_{a}\subseteq\Omega$ centered at $a$ (which exists since $\Omega$ has Lipschitz boundary)  such that   $v \ne 0$ on $S_a$. Making use of a similar argument as in the proof of Theorem \ref{thm3.4}, with $B_1$ replaced by $S_a$, one can  obtain $\nabla \log |v|\notin L_{loc}^n(\Omega)$.
\end{proof}

A natural way to view Theorem \ref{t1} is as follows. Denote by $Lip(\Omega)$ the set of all Lipschitz functions on $\Omega$.   Theorem \ref{t1} implies that for each $a\in {\Omega}$, $$T_a(Lip(\Omega))\cap  W^{1,n}(\Omega)=\emptyset,$$
where $T_a$ is a  (non-linear) map  on $ Lip(\Omega)$  defined by   $T_a(u)=\log |u-u(a)|$,  $u\in Lip(\Omega)$.  The following Corollary is a direct consequence  of Theorem \ref{t1} as well.

\begin{corollary}Let $\Omega$ be an open set in $\mathbb{R}^n$  and   $u: \Omega\rightarrow \mathbb R$ be    locally Lipschitz   on $\Omega$. If the zero set $Z_u$ of $u$ is neither empty nor $\Omega$,  and
\begin{equation}\label{po}
\int_{\Omega\setminus Z_u}|\nabla \log|u(x)||^p\, dv<\infty.
\end{equation}
then  $p<n$.  \boxx
\end{corollary}


\begin{proposition}\label{snf}
Let $\Omega$ be an open set in $\mathbb{R}^n, n\ge 2$, and   $\phi:\Omega\to\re$  with $\phi\in W_{loc}^{1,n}(\Omega)$. The following statements hold for the exponential $e^{-|\phi|}$.
\begin{enumerate}
   \item $e^{-|\phi|}$  vanishes to finite order in the $L^2$ sense anywhere in $\Omega$. 
   \item If     $e^{-|\phi|}$  is locally Lipschitz on $\Omega$, then $e^{-|\phi|}$ is nowhere zero on $ \Omega$. $e^{-|\phi|}$ is nowhere zero on $ \overline\Omega$ if in addition $\Omega$ has Lipschitz boundary.  
\end{enumerate}
\end{proposition}
\begin{proof}
Since $\phi\in W_{loc}^{1,n}(\Omega)$, we have $|\phi|\in W_{loc}^{1, n}(\Omega)$ as well. 
 The function $u=e^{-|\phi|}$ satisfies $|u|<1$ and
$$|\nabla u| = |\nabla |\phi|| e^{-|\phi|}= \left|\nabla |\phi|\right| |u|\le  \left|\nabla |\phi|\right| \in L_{loc}^n(\Omega).$$
See for instance \cite[pp. 308]{Ev}. Hence   $u \in W_{loc}^{1, n}(\Omega)$ and satisfies $|\nabla u| =   V|u|$ with $V= |\nabla |\phi||\in L_{loc}^n(\Omega)$. By Theorem \ref{main}, $u$ cannot vanish to infinite order in the $L^2$ sense anywhere in $\Omega$. 
   
 If  $e^{-|\phi|}$  is also Lipschitz on $\Omega$, and   $e^{-|\phi|}$ is zero at $x_0\in  \Omega$, then $|\phi| = -\log |u-u(x_0)|\notin W_{loc}^{1,n}(\Omega)$ by Theorem \ref{t1}. Contradiction!
\end{proof}

Before proving Theorem \ref{nf}, let us recall the \noindent\textbf{Moser-Trudinger inequality: }Let $\Omega$ be a bounded domain in $\mathbb R^n$ with Lipschitz boundary, and $\alpha_n = n w_{n-1}^\frac{1}{n-1}$ where $w_{n-1}$ is the surface area of the unit sphere in $\mathbb R^n$. There exists a positive constant  $C_{MT}$ depending only on $n$ such that $ $
$$ \sup_{u\in W_0^{1, n}(\Omega),\ \  \|\nabla u\|_{L^n(\Omega)}\le 1 }\int_\Omega e^{\alpha_n |u(x)|^{\frac{n}{n-1}}} dv\le C_{MT} |\Omega|.$$
Here $ |\Omega|$ is the volume of $\Omega$. We shall use the Moser-Trudinger inequality to prove that the exponential of $W^{1, n}$ functions is $L^2$ integrable. 
\medskip

\begin{proof}[Proof of Theorem \ref{nf}: ] First  we show that $e^\phi\in L^2_{loc}(\Omega)$.  Pick $r $ small enough such that $ B_{2r}(x_0)\subseteq\Omega$. By Sobolev extension theorem, there exists an extension $ \tilde \phi \in W_0^{1, n}(B_{2r}(x_0)) $ of $\phi|_{B_r(x_0)}$   such that 
$$ a= \|\nabla \tilde \phi\|_{  L^n(B_{2r}(x_0))  )}\le C \| \phi\|_{  W^{1, n}(B_{r}(x_0)))}  $$
for some constant $C$ dependent only on $r $ and $n$.  In particular,   $\tilde \phi_1:  = a^{-1}\tilde \phi \in W_0^{1, n}(B_{2r}(x_0))$ and $   \|\nabla  \tilde \phi_1\|_{L^n(B_{2r}(x_0))}\le 1.$ Thus 
  one applies the Moser-Trudinger inequality to obtain   $$  \int_{B_{2r}(x_0)}  e^{\alpha_n |\tilde \phi_1(x)|^{\frac{n}{n-1}}}  dv\lesssim 1.$$
Noting that $2\tilde\phi < \alpha_n |\tilde \phi_1|^{\frac{n}{n-1}} $ when $ |\tilde \phi|> 2^{n-1}a^n \alpha_n^{1-n} $, we further have
$$ \int_{B_{2r}(x_0)\cap\{|\tilde \phi|> 2^{n-1}a^n \alpha_n^{1-n}\}} e^{2\tilde\phi(x)}\ dv\le  \int_{B_{2r}(x_0)}  e^{\alpha_n |\tilde \phi_1(x)|^{\frac{n}{n-1}}} \ dv\lesssim 1.  $$
The claim that  $e^\phi\in L^2_{loc}(\Omega)$ is thus a consequence of the following inequality 
\begin{eqnarray*}
         \int_{B_r(x_0)} e^{2\phi(x)}\ dv &\le& \int_{B_{2r}(x_0)\cap\{|\tilde \phi|\le 2^{n-1}a^n \alpha_n^{1-n}\}} e^{2\tilde \phi(x)}\ dv+ \int_{B_{2r}(x_0)\cap\{|\tilde \phi|> 2^{n-1}a^n \alpha_n^{1-n}\}} e^{2\tilde\phi(x)}\ dv\\
         &\lesssim& e^{2^{n}a^n \alpha_n^{1-n} }r^n + 1. 
\end{eqnarray*} 
 On the other hand, by Proposition \ref{snf} part {\it (1)}, $e^{-|\phi|} $ vanishes to finite order in the $L^2$ sense at $x_0$. Equivalently, there exists some $m_0\ge 0$ such that 
 $$ \overline{\lim_{r\rightarrow 0}}\  r^{-m_0}\int_{|x-x_0|<r}|e^{-|\phi(x)|}|^2 dv  >0.$$
 Since  $e^\phi\ge e^{-|\phi|}\ge 0$ and $ e^\phi\in L^2_{loc}(\Omega)$, we further have
 $$ \overline{\lim_{r\rightarrow 0}}\ r^{-m_0}\int_{|x-x_0|<r}|e^{\phi(x)}|^2 dv\ge  \overline{\lim_{r\rightarrow 0}}\ r^{-m_0}\int_{|x-x_0|<r}|e^{-|\phi(x)|}|^2 dv  >0. $$
 Namely, $ e^\phi$ vanishes to finite order in the $L^2$ sense at $x_0$.
\end{proof}

 \begin{corollary}\label{nf1}
 Let $\Omega$ be an open set in $\mathbb{R}^n, n\ge 2$. Suppose   $\phi:\Omega\to\re$  with $\phi\in W_{loc}^{1,n}(\Omega)$. If  $e^\phi$ is Lipschitz on $\Omega$, then $e^\phi$ is nowhere zero on $ \Omega$.
\end{corollary}
\begin{proof}
  It is not hard to verify that for all $x_1, x_2\in \Omega$,
 $$\left|e^{-|\phi(x_2)|} - e^{-|\phi(x_1)|}\right|\le |e^{ \phi(x_1) } - e^{ \phi(x_2)}|.  $$
In particular, $e^{-|\phi|}$ is Lipschitz whenever $e^{\phi}$ is so. Applying  Proposition \ref{snf} part {\it (2)}, we have   $e^{-|\phi|}$, and thus $e^{\phi}$, is nowhere zero on $\Omega$.
\end{proof}

\section{In relation to $\bar\partial$}

On domains in $\mathbb C^n$,  if the gradient operator $\nabla$ is replaced by the $\bar\partial$ operator, then Theorem \ref{thm1.1} fails, even for real analytic functions. In fact, there are  real analytic functions that vanish to any given order at one point and satisfy $ |\bar\partial u| \le V |u|$ for some $V\in L^\infty$.

\begin{example}\label{ex5}
Let $f$ be a holomorphic function on $B_1\subseteq\mathbb C^n$ that vanishes to order $k$ at $0$, $k\ge 1$. Letting $u(z) =\left(1+\frac{\bar z_1}{2}\right)f(z)$, then $u$ is real analytic on $B_1$, vanishes to order $k$ at $0$ and  satisfies $|\bar\partial u| \le 4 |u|$.
\end{example}

On the other hand,    since $ |\nabla u|^2= |\partial u|^2+|\bar\partial u|^2  $ for a Lipschitz  $u$, by Theorem \ref{thm1.1} we have  near any neighborhood $U$ of a zero point in $\partial Z_u$ of $u$,
$$ \int_U\frac{|\nabla u(z)|^2}{|u(z)|^2}\ dv = \int_U\frac{|\bar\partial u(z)|^2}{|u(z)|^2} + \frac{| \partial u(z)|^2}{|u(z)|^2}\ dv =\infty. $$
The following propositions discuss a finer property about the $L^2$ divergence of $\frac{\nabla u}{u}$ concerning the smooth extension of holomorphic functions beyond the boundary. In particular, they exhibit an intrinsic obstruction for  holomorphic functions to be extended smoothly across the boundary. We note  that for smooth functions, the flatness in the $L^2$ sense at a point is equivalent to the vanishing of all jets at that point.

\begin{proposition}\label{pr}
    Let $\Omega$ be a domain in $\mathbb C$ and $z_0\in \partial\Omega$. Let $u$ be a nonconstant holomorphic function on $\Omega$. If $u$ can be extended smoothly across $z_0$, still denoted by $u$, and  $u(z_0) =0$, then there exists a neighborhood $U$ of $z_0$ such that one of the following  holds.
    \begin{enumerate}
        \item If $u$ vanishes to finite order at $z_0$, then
        \begin{equation}\label{jj}
            \int_U\frac{|\bar\partial u(z)|^2}{|u(z)|^2}\ dv <\infty \ \ \text{and}\ \ \int_U \frac{| \partial u(z)|^2}{|u(z)|^2}\ dv =\infty.
        \end{equation}

        \item If $u$ vanishes to infinite order at $z_0$, then
        \begin{equation}\label{pp}
            \int_U\frac{|\bar\partial u(z)|^2}{|u(z)|^2}\ dv =\infty.
        \end{equation}
    \end{enumerate}
\end{proposition}
\begin{proof}
   Without loss of generality let $z_0=0$.  In  {\it (1)}, since  $u$ vanishes to finite order at $0$ and is holomorphic on $\Omega$, $u = cz^k +O(|z|^{k+1})$ near $0$ for some constant $c\ne 0, k\in \mathbb Z^+$. With  a direct computation we have
   \begin{equation}\label{eqa}
        \frac{| \bar\partial u| }{|u| } = \frac{O(z^k)}{|cz^k+ O(|z|^{k+1})|}  =O(1)  \ \ \text{and} \ \ \frac{| \partial u| }{|u| } = \frac{k}{|z|} + O (1), 
   \end{equation}
   from which \eqref{jj} follows.

   For {\it (2)}, if not, then set  $V = \frac{\bar\partial u}{u}$ where  $u\ne 0$, and $V=0$ otherwise on $U$, so that $V\in L^2(U)$ and $\bar\partial u = Vu$ on $U$. According to Theorem \ref{pz}, since $u$ is flat at $z_0$, we  have $u\equiv 0$ on $U$. In particular, $u=0$ on the open set $U\cap \Omega$. By the holomorphic property of $u$ on $\Omega$, we further have $u\equiv 0$ on $\Omega$. This contradicts the assumption that $u$ is nonconstant on $\Omega$.
\end{proof}

The following two corollaries give alternative characterizations on the vanishing order of smooth extension of holomorphic functions across the boundary. 

\begin{corollary}\label{cor1}
     Let $\Omega$ be a domain in $\mathbb C$ and $z_0\in \partial\Omega$. Let $u$ be a nonconstant holomorphic function on $\Omega$, and extend smoothly across $z_0$, still denoted by $u$, with $u(z_0) =0$.  Then the following statements are equivalent to each other.
     \begin{enumerate}
         \item $u$ vanishes to finite order at $z_0$.
         \item   $    \frac{  |\bar\partial u|}{|u|}\in L^\infty$ near $z_0$.
    \item    $  \frac{        |\bar\partial u|}{|u|}\in L^2        $ near $z_0$.
     \end{enumerate}
\end{corollary}
\begin{corollary}\label{cor2}
     Let $\Omega$ be a domain in $\mathbb C$ and $z_0\in \partial\Omega$. Let $u$ be a nonconstant holomorphic function on $\Omega$, and extend smoothly across $z_0$, still denoted by $u$, with $u(z_0) =0$.  Then the following statements are equivalent to each other.
     \begin{enumerate}
         \item $u$ vanishes to infinite order at $z_0$.
         \item       $    \frac{        |\bar\partial u|}{|u|}\notin L^\infty$ near $z_0$.    
    \item       $  \frac{        |\bar\partial u|}{|u|}\notin L^2        $ near $z_0$.
     \end{enumerate}
\end{corollary}
\begin{proof}[Proof of Corollary \ref{cor1} and \ref{cor2}:]
 For Corollary \ref{cor1},   {\it (2)} $\Rightarrow$  {\it (3)} is trivial.  {\it (3)} $\Leftrightarrow $ {\it (1)} is a direct consequence of Proposition  \ref{pr}. {\it (1)} $\Rightarrow$  {\it (2)} follows from \eqref{eqa} in the proof of Proposition  \ref{pr}.  Corollary \ref{cor2} can be proved similarly.
\end{proof}


\begin{example}
Let $\mathbb H^+$ be the upper half plane in $\mathbb C$. The function $$u = \exp\left({\displaystyle{\frac{1}{i\sqrt{iz}}}}\right), \ \ \ \arg iz\in (\frac{\pi}{2}, \frac{3\pi}{2}),$$ is holomorphic on $\mathbb H^+$ and vanishes to infinite order at $z_0=0$. It allows for a smooth extension across $0$.  By Proposition \ref{pr} {\it (2)}, every smooth extension of $u$ on a neighborhood $U$ of $0$ should satisfy \eqref{pp}. Note that $u$ cannot  extend holomorphically across $0$.

For every $k\ge 1$, the function $u = z^k $ is holomorphic on $\mathbb H^+$  and vanishes to finite order $k$ at $0$. By Proposition \ref{pr} {\it (1)}, every smooth extension of $u$ on a neighborhood $U$ of $0$ should satisfy \eqref{jj}. For a less trivial example  towards Proposition \ref{pr} {\it (1)}   without holomorphic extension across $0$, one can consider  $u = z^k + e^{\frac{1}{i\sqrt{iz}}}$ on $\mathbb H^+$   instead, and obtain \eqref{jj} for every smooth extension of $u$ across $0$.
\end{example}

\begin{proposition}\label{prn}
    Let $\Omega$ be a domain in $\mathbb C^n$ and $z_0\in \partial\Omega$. Let $u$ be a nonconstant holomorphic function on $\Omega$. If $u$ can be extended smoothly across $z_0$, still denoted by $u$, and  $u(z_0) =0$, then there exists a neighborhood $U$ of $z_0$ such that one of the following  holds.
    \begin{enumerate}
        \item If $u$ vanishes to finite order at $z_0$, then there exists a complex line $L$ passing through $z_0$ such that
        \begin{equation}\label{kk}
               \int_{U\cap L} \frac{| \partial u(z)|^2}{|u(z)|^2}\ dv =\infty.
        \end{equation}

        \item If $u$ vanishes to infinite order at $z_0$, then for every complex line $L$ passing through $z_0$, 
        \begin{equation}\label{mm}
            \int_{U\cap L}\frac{|\bar\partial u(z)|^2}{|u(z)|^2}\ dv =\infty.
        \end{equation}
    \end{enumerate}
\end{proposition}
\begin{proof}
    For simplicity let $z_0=0$ and $n=2$. The higher dimensional cases can be proved similarly. If  $u$ vanishes to finite order at $0$, then after a holomorphic change of coordinates, there exists some $k\in \mathbb Z^+$ such that $$u = z_1^k + g_{k-1}(z_2)z_1^{k-1}+\cdots +g_0(z_2) + h(z)$$  near $0$. Here for each $j = 0, \ldots, k-1$, $g_j$ is  smooth   on $U$, holomorphic on $\Omega\cap U$ and $g_j(0)=0$, and $h$ is a smooth  function on $U$ with $h=0$ on $\Omega\cap U$. In particular, $h$ is flat at $0$.
    Thus on the complex line $L = \{(z_1, 0)\in \mathbb C^2\}$, we have $u|_{U\cap L} = z_1^k+h(z_1, 0)$ and so
    $$ \frac{| \partial_{z_1} u| }{|u| } = \frac{k}{|z_1|} + O (1). $$
    Hence \eqref{kk} holds.

  Assume  $u$ vanishes to infinite order at $0$ and there exists a complex line $L$ through $0$ such that  $\frac{|\bar\partial u| }{|u| }\in L^2(U\cap L )$. Applying a holomorphic change of coordinates if necessary, one can always write $L = \{(z_1,  0)\in \mathbb C^2\}$. Then $v:  = u|_{U\cap L}$  vanishes to infinite order at $0$ and 
$$\frac{|\bar\partial_{z_1} u| }{|u| }\le \frac{|\bar\partial u| }{|u| }\in L^2(U\cap L ).$$ In particular, there exists some $W\in L^2(U\cap L)$ such that   $\bar\partial_{z_1} v =Wv$ on $U\cap L$. By Theorem \ref{pz}, we have $v\equiv 0$. Thus \eqref{mm} holds.
\end{proof}

\begin{proposition}\label{pop}
    Let $\Omega$ be a domain in $\mathbb C^n$ and $z_0\in \partial\Omega$. Let $u$ be a function holomorphic on $\Omega$ and   smooth  on a neighborhood $U\subseteq\mathbb C^n$ of $z_0$. Then one of the following mutually exclusive cases holds.
    \begin{enumerate}
        \item  $u$ is holomorphic on $\Omega\cup U$.
        \item  \begin{equation}\label{e12}
        \int_{U\setminus Z_{\bar\partial u}} \frac{\sum_{j, k=1}^n|\bar\partial^2_{z_jz_k}u(z)|^2}{\sum_{j=1}^n|\bar\partial_{z_j} u(z)|^2} \ dv=\infty,
    \end{equation}
    where  $Z_{\bar\partial u}$ is the zero set  of the vector function $\bar\partial u$. 
    \end{enumerate}
\end{proposition}
\begin{proof}
Suppose $u$ is not holomorphic on $\Omega\cup U$ and \eqref{e12} fails. Then  $Z_{\bar\partial u}\cap U\ne U$ and the function
  \begin{equation*}\label{lam}
     W= \begin{cases}
 \sqrt{\frac{\sum_{j, k=1}^n|\bar\partial^2_{z_jz_k}u|^2}{ \sum_{j=1}^n|\bar\partial_{z_j} u|^2}}, &\text{on}\ \ U\setminus Z_{\bar\partial u}; \\
0, &\text{on}\ \ Z_{\bar\partial u}
\end{cases} 
\end{equation*}
 belongs to $L^2_{loc}(U)$. 
  Let $v= (\bar\partial_{z_1} u, \ldots, \bar\partial_{z_n} u)$. Then $v: U\rightarrow \mathbb C^n$ satisfies $|\bar\partial v| =W|v|$ on $U$ and vanishes on the nonempty open set $U\cap \Omega$. According to the weak unique continuation property in \cite[Theorem 1.2]{PZ}, we have $v\equiv 0$ on $U$, contradicting the assumption that $Z_{\bar\partial u}\cap U\ne U$.
\end{proof}

We point out that in the case when $\Omega$ is pseudoconvex  with smooth boundary, there always exists a  function which is  holomorphic  on $\Omega$ and smooth on $\overline\Omega$, but  does not extend  holomorphically across a boundary point $z_0$. Thus  for every smooth extension of this function, part {\it (2)} of Proposition \ref{pop} always occurs.

\begin{remark}
While all the propositions and corollaries in this section are formulated for holomorphic functions with smooth extension across a boundary point, the same reasoning and conclusions can be extended without effort to more general settings, including formally holomorphic functions -- smooth functions where the Taylor expansion at that point does not contain $\bar z$ terms. See \cite
{FP} for more discussion on formally holomorphic functions. 
\end{remark}

\subsubsection*{Acknowledgments}
  Part of this work by the second named author was conducted while on
  sabbatical leave visiting Huaqiao university in China in Spring 2024.  He thanks Jianfei Wang for his invitation, and 
  the host institution for its hospitality and excellent research
  environment.
  
  We also acknowledge the helpful comments of the reviewers, whose suggestions improved the writing and results of Sections \ref{sec3} and \ref{sec5}.

\end{document}